\documentclass[12pt,reqno]{amsart}
 
\usepackage{amsmath,amsthm,amssymb,mathrsfs,amscd,epsfig,color}
\usepackage{url}
\usepackage{graphicx}
\usepackage{mathrsfs}
\usepackage{ulem}
\usepackage{fancybox}
\usepackage{pb-diagram} 
\usepackage[all]{xy}
\usepackage{comment}
\usepackage{mathtools}
\usepackage{centernot}
\usepackage{placeins} % for \FloatBarrier 
\usepackage{url}

\makeatletter

\makeatletter
    \newcommand\contFrac{\@ifstar{\@contFracStar}{\@contFracNoStar}}

   \def\singleContFrac#1#2{%
        \begin{array}{@{}c@{}}%
            \multicolumn{1}{c|}{#1}%
            \\%
            \hline%
           \multicolumn{1}{|c}{#2}%
        \end{array}%
   }

    \def\@contFracNoStar#1{%
        \mathchoice{% * Display style
            \@contFracNoStarDisplay@#1//\@nil%
        }{%           * Text style
            \@contFracNoStarInline@#1//\@nil%
        }{%           * Script style
            \@contFracNoStarInline@#1//\@nil%
        }{%           * Script script style
            \@contFracNoStarInline@#1//\@nil%
        }%
    }

    \def\@contFracNoStarDisplay@#1//#2\@nil{%
        \@ifmtarg{#2}{%
            #1%
        }{%
            #1+\cfrac{1}{\@contFracNoStarDisplay@#2\@nil}%
        }%
    }

        \def\@contFracNoStarInline@#1//#2\@nil{%
            \@ifmtarg{#2}{%
                #1%
            }{%
                #1 \@@contFracNoStarInline@@#2\@nil%
            }%
        }
        \def\@@contFracNoStarInline@@#1//#2\@nil{%
            \@ifmtarg{#2}{%
                + \singleContFrac{1}{#1}%
            }{%
                + \singleContFrac{1}{#1} \@@contFracNoStarInline@@#2\@nil%
            }%
        }

    \def\@contFracStar#1{%
        \mathchoice{% * Display style
            \@contFracStarDisplay@#1////\@nil%
        }{%           * Text style
            \@contFracStarInline@#1//\@nil%
        }{%           * Script style
            \@contFracStarInline@#1//\@nil%
        }{%           * Script script style
            \@contFracStarInline@#1//\@nil%
        }%
    }

    \def\@contFracStarDisplay@#1//#2//#3\@nil{%
        \@ifmtarg{#2}{%
            #1%
        }{%
            #1 + \cfrac{#2}{\@contFracStarDisplay@#3\@nil}%
        }%
    }

        \def\@contFracStarInline@#1//#2\@nil{%
            \@ifmtarg{#2}{%
                #1%
            }{%
                #1 \@@contFracStarInline@@#2\@nil%
            }%
        }
        \def\@@contFracStarInline@@#1//#2//#3\@nil{%
            \@ifmtarg{#3}{%
                + \singleContFrac{#1}{#2}%
            }{%
                + \singleContFrac{#1}{#2} \@@contFracStarInline@@#3\@nil%
            }%
        }
\makeatother

       \numberwithin{equation}{section}
\theoremstyle{plain}
\newtheorem*{thmA}{Main Theorem}

\newtheorem{thm}{Theorem}[section]
\newtheorem{lem}[thm]{Lemma}
\newtheorem{cor}[thm]{Corollary}

\newtheorem{pro}[thm]{Proposition}
\theoremstyle{definition}

\newtheorem{rem}[thm]{Remark}

\newtheorem*{prf*}{Proof}
\newtheorem*{cla*}{Claim}

\newtheorem*{pf*}{}
\newtheorem*{lem*}{LemmaA}
\newtheorem*{lm*}{LemmaB}
\makeatletter
\@namedef{subjclassname@2020}{%
  \textup{2020} Mathematics Subject Classification}
\makeatother
\title[Fractal transference principle for Laurent Series]
{Fractal transference principle \\ for continued fractions of Laurent series}
\author{Yuto Nakajima }

\address{Faculty of Science and Engineering, Doshisha University, Kyoto, 610-0394, JAPAN}
\email{yunakaji@mail.doshisha.ac.jp}

\begin{document}

\begin{abstract} 

We establish a fractal transference principle for continued fraction
expansions over the field of Laurent series. Let $S$ be an
infinite subset of the set of all polynomials over a finite field of $q$ elements of positive degree with growth density exponent $\alpha \ge 1$, and let $U \subset S$ be a
subset of positive relative upper density. We prove that there exists a subset
$E_{S,U}$ of the set of points whose continued fraction digits are pairwise distinct
and belong to $S$ such that
\[\dim_{\rm H}
E_{S,U}=\frac{1}{2\alpha}.\]
Moreover, the set of digits appearing in the continued fraction expansions of
points in $E_{S,U}$ recovers the relative upper density of $U$ in $S$. We also show
that the same construction preserves the relative upper density of the
corresponding degree sets in $\mathbb N$. As a consequence, combinatorial
statements for subsets of $\mathbb N$ of positive upper density can be transferred
to degree sets arising from continued fraction expansions of Laurent series on sets
of optimal Hausdorff dimension.
\end{abstract}

\maketitle

\section{Introduction}\label{intro}
The existence of patterns, especially arithmetic progressions (APs), in subsets of the natural numbers has long been a central theme. 
A foundational result in this direction is Van der Waerden’s theorem \cite{Wae27}, which was later strengthened by Szemer\'edi \cite{S75}, who proved that every subset of $\mathbb{N}$ with positive upper density contains arithmetic progressions of arbitrary length. 
Furstenberg \cite{Fur77,Fur81} reinterpreted this phenomenon in ergodic theoretic terms via the correspondence principle and multiple recurrence, thereby giving a new proof of Szemer\'edi’s theorem. More remarkably, Green and Tao \cite{GT08} showed that even the prime numbers contain arithmetic progressions of arbitrary length.

Within fractal dimension theory, analogous questions have been investigated through number-theoretic expansions. In the setting of regular continued fractions, each irrational number $x \in (0,1)\setminus\mathbb{Q}$ admits an expansion
\[
x=\cfrac{1}{a_{1}(x)+\cfrac{1}{a_{2}(x)+\cfrac{1 }{a_{3}(x)+\ddots}}},\]
where $a_n(x)\in\mathbb{N}$ for $n\ge 1$ are called the partial quotients. Define
\[J=\{x\in(0,1)\setminus\mathbb{Q}\colon a_n(x)<a_{n+1}(x)\ \text{for every } n\ge 1\}.\]
Since the map
$x \mapsto \{a_n(x)\colon n\in\mathbb{N}\}$
induces a bijection between $J$ and the family of infinite subsets of $\mathbb{N}$, the set $J$ provides a natural framework in which to study combinatorial properties of subsets of $\mathbb{N}$ via continued fractions.
Ramharter \cite{R85} proved that $J$ has Hausdorff dimension $1/2$, refining Good’s classical theorem \cite[Theorem~1]{G}, which states that the set
\[\{x\in(0,1)\setminus\mathbb{Q}\colon a_n(x)\to\infty\ \text{ as } \ n\to\infty\}\]
also has Hausdorff dimension $1/2$. Tong and Wang \cite{TW} initiated the study of APs in continued fractions by proving that the set of points $x\in J$ for which the set $\{a_n(x)\colon n\in\mathbb{N}\}$ of partial quotients contains arithmetic progressions of arbitrary length and arbitrary common difference has Hausdorff dimension $1/2$, the same as that of $J$. Subsequently, Zhang and Cao \cite{ZC2} proved that the set of points $x\in J$ for which the set of partial quotients has upper Banach density $1$ has Hausdorff dimension $1/2$. We also refer to \cite{NTW} for quantitative results concerning arithmetic progressions in continued fractions.

Motivated by these works, the author and Takahasi \cite{NT} considered the larger set
\[E=\{x\in(0,1)\setminus\mathbb{Q}\colon a_m(x)\ne a_n(x)\ \text{for all } m\ne n\}.\]
For each $x\in E$, the set $\{a_n(x)\colon n\in\mathbb{N}\}$ can still be viewed naturally as an infinite subset of $\mathbb{N}$. Let $(\ast)$ denote a combinatorial property that is satisfied by every subset of $\mathbb{N}$ with positive upper density, such as the existence of APs of arbitrary length \cite{S75}. The {\it fractal transference principle} asserts that for every subset $S\subset\mathbb{N}$ with positive upper density, there exists a set $E_S\subset E$ of Hausdorff dimension $1/2$ such that, for every $x\in E_S$, the intersection $\{a_n(x)\colon n\in\mathbb{N}\}\cap S$ inherits the property $(\ast)$. In this way, the principle builds a bridge between density combinatorics and the fractal dimension theory of continued fractions. For the fractal transference principle for other kinds of expansions, see \cite{GOQZ, NT2}. The aim of this paper is to establish an analogue of this principle for continued fraction expansions over the field of Laurent series.

\subsection{Continued fractions of Laurent series}
 We now recall the continued fraction
expansion over the field of Laurent series and fix the notation used throughout the
paper. Let $\mathbb{F}_q$ be a finite field of $q$ elements, and let
$\mathbb{F}_q((X^{-1}))$ denote the field of formal Laurent series, namely
\[
\mathbb{F}_q((X^{-1}))=
\left\{
\sum_{i=n_0}^{\infty} c_i X^{-i} \colon n_0\in\mathbb{Z},\ c_i\in\mathbb{F}_q
\right\}.
\]

Let $\mathbb{F}_q[X]$ denote the ring of polynomials over $\mathbb{F}_q$ and
let $\mathbb{F}_q(X)$ denote the field of rational functions over $\mathbb{F}_q$,
that is, the field of fractions of $\mathbb{F}_q[X]$. Equivalently,
\[\mathbb{F}_q(X)=
\left\{\frac{P(X)}{Q(X)} \colon P(X),Q(X)\in \mathbb{F}_q[X], Q(X)\neq 0\right\}.\]
Identifying each rational function with its Laurent expansion in $X^{-1}$ at infinity, we regard $\mathbb{F}_q(X)$ as a subfield of
$\mathbb{F}_q((X^{-1}))$.

For a nonzero element $x\in \mathbb{F}_q((X^{-1}))$, define $\deg x={-\inf\{n\in\mathbb{Z}\colon c_n\neq 0\}}$ and 
\[
\|x\|_q:=q^{\deg x}, \ \|0\|_q:=0.
\]
Then $\|\cdot\|_q$ is a non-Archimedean absolute value on
$\mathbb{F}_q((X^{-1}))$, and $\mathbb{F}_q((X^{-1}))$ is complete with respect to the induced metric. Let
\[\Delta=\{x\in \mathbb{F}_q((X^{-1}))\colon \|x\|_q<1\}.\]

For
$x=\sum_{i=n_0}^{\infty} c_i X^{-i}\in \mathbb{F}_q((X^{-1})),$
write
$[x]:=\sum_{i=n_0}^{0} c_i X^{-i}$
for the integral part of $x$. Following Artin \cite{Ar}, define the map $T:\Delta\to \Delta$ by
\[
T(x):=\frac{1}{x}-\left[\frac{1}{x}\right]\ (x\neq 0),\ T(0):=0.
\]
This map gives the regular continued fraction algorithm over the field of Laurent series.  Let 
\[
\mathbb{F}_q^{1}[X]=
\left\{
A\in \mathbb F_q[X] \colon \deg A\ge 1
\right\}.
\]
For $x\in \Delta$ define 
\begin{equation}\label{LSCF}A_n(x)=\left[\frac{1}{T^{n-1}(x)}\right]\in \mathbb{F}_q^{1}[X] \ \text{if}\ T^{n-1}(x)\neq 0\end{equation}
for each $n\ge 1.$
We call $A_n(x)$ the $n$th {\it digit} of $x$ whenever it is defined.
Then every $x\in \Delta \setminus \mathbb F_q(X)$ admits a unique infinite continued fraction expansion of Laurent series
\begin{equation}\label{LSCF2}x=\cfrac{1}{A_{1}(x)+\cfrac{1}{A_{2}(x)+\cfrac{1 }{A_{3}(x)+\ddots}}}.\end{equation}
On the other hand, the continued fraction expansion of $x$ is finite if and only if $x \in \Delta \cap \mathbb{F}_q(X)$.
See \cite{BN} for more details.

Metric and fractal properties of these expansions have been studied extensively;
see \cite{BN, FSZ24, HH, HHY, HS, HW, HWWY, Wu, Wu2}. In the specific direction of APs, Hu and Hu \cite{HH} proved that the set of points $x\in \Delta \setminus \mathbb F_q(X)$ whose degree sequence $(\deg A_n(x))_{n=1}^{\infty}$ is strictly increasing and whose degree set
$\{\deg A_n(x)\colon n\in\mathbb{N}\}$
contains arithmetic progressions of arbitrary length and arbitrary common difference has Hausdorff dimension $1/2$. This is the Laurent series analogue of \cite{TW}.

\subsection{Main result}\label{state-sec}
We now state the main result of this paper. In order to formulate it, we need to
distinguish between two notions of density. One is a density on subsets of
$\mathbb F_q^1[X]$, measured by counting polynomials of degree at most $N$.
The other is a density on subsets of $\mathbb N$, which will be applied to the
corresponding degree sets. The theorem below shows that our construction
simultaneously preserves both kinds of density.

Given two infinite sets $B \subset G \subset \mathbb N$, define the upper density of
$B$ relative to $G$ by
\[
\overline{d}(B | G)=\limsup_{N\to\infty}\frac{\#(B\cap [1,N])}{\#(G\cap [1,N])}.
\]
For $N \in \mathbb N$ and $S \subset \mathbb F_q^1[X]$, set
\[
Q_N(S)=S\cap \mathbb F_q^{1,N}[X],
\]
where $\mathbb F_q^{1,N}[X]:=\{A\in \mathbb F_q^{1}[X]\colon \deg A\le N\}.$
Given two infinite sets $U \subset S \subset \mathbb F_q^1[X]$, define the upper
density of $U$ relative to $S$ by
\[
\overline{d}_q(U| S)=\limsup_{N\to\infty}\frac{\#Q_N(U)}
{\#Q_N(S)}.
\]
When $G=\mathbb N$, we simply write $\overline{d}(B):=\overline{d}(B| \mathbb N)$.
Likewise, when $S=\mathbb F_q^1[X]$, we write
$\overline{d}_q(U):=\overline{d}_q(U| \mathbb F_q^1[X])$.

We say that $S \subset \mathbb F_q^1[X]$ has {\it growth density} if there exist constants $\gamma_1,\gamma_2>0$,
$\alpha\ge 1$, and $\beta\ge 0$ such that for all sufficiently large $N$,
\[
\gamma_1 \frac{q^{N/\alpha}}{N^\beta}
\le \# Q_N(S)
\le
\gamma_2 \frac{q^{N/\alpha}}{N^\beta}.
\]
In this case, we also say that $S$ has growth density with exponent $\alpha$.

For $S\subset \mathbb F_q^1[X]$, write
\[
\deg(S):=\{\deg A\colon A\in S\}.
\]
We also set
\[
E=\{x\in \Delta \setminus \mathbb F_q(X)\colon A_m(x)\neq A_n(x)\text{ for all }m\neq n\}.
\]
Let $\dim_{\rm H}$ denote the Hausdorff dimension on $\Delta$.
The following theorem is the main result of the paper.

\begin{thmA}
Let $S$ be an infinite subset of $\mathbb F_q^1[X]$ that has growth density with exponent $\alpha\geq1$.
For any $U\subset S$ with $\overline{d}_q(U|S)>0$, 
there exists a subset $E_{S,U}$ of $\{x\in E \colon \{A_n(x)\colon n\in\mathbb N\}\subset  S \}$ such that \begin{equation*}\dim_{\rm H}E_{S,U}=\dim_{\rm H}\{x\in E\colon \{A_n(x)\colon n\in\mathbb N\}\subset S\}=\frac{1}{2\alpha},\end{equation*} 
\begin{equation*}\overline{d}_q\left(\bigcup_{n\in\mathbb N}\bigcap_{x\in E_{S,U}}\{A_n(x)\}\cap U\middle|S\right)=\overline{d}_q(U|S),\end{equation*}
and 
\begin{equation*}\overline{d}\left(\bigcup_{n\in\mathbb N}\bigcap_{x\in E_{S,U}}\{\deg A_n(x)\}\cap \deg(U)\middle|\deg(S)\right)=\overline{d}(\deg(U)|\deg(S)).\end{equation*}
In particular, 
\begin{equation*}\dim_{\rm H}\left\{
\begin{tabular}{l}
\!\!\!$x\in E\colon$\\
\!\!\!$\{A_n(x)\colon n\in\mathbb N\}\subset S$,\!\!\!\\
\!\!\!$\overline{d}_q(\{A_n(x)\colon n\in\mathbb N\}\cap U |S)=\overline{d}_q(U|S),$
\!\!\!\\
\!\!\!$\overline{d}(\{\deg A_n(x)\colon n\in\mathbb N\}\cap \deg(U) |\deg(S))$\!\!\!\\
\!\!\!$=\overline{d}(\deg(U)|\deg(S))$\!\!\!
\end{tabular}
\right\}=\dim_{\rm H}
\left\{\begin{tabular}{l}
\!\!\!$x\in E\colon$\\
\!\!\!$\{A_n(x)\colon n\in\mathbb N\}\subset S$\!\!\!\end{tabular}
\right\}.\end{equation*}
\end{thmA}

Main Theorem  asserts
that for every relatively dense subset $U$ of a set $S$, one can
construct a large fractal set of $\Delta\setminus \mathbb F_q(X)$ whose digits lie in $S$ and
which recovers both the relative density of $U$ and the relative density of
its degree set.
As a direct consequence of Main Theorem, we obtain the following corollary by
specializing to the full digit set $S=\mathbb F_q^1[X]$.

\begin{cor}\label{cor-FS}
Let $S\subset\mathbb F_q^1[X]$. 
 If $\overline{d}_q(S)>0$, then there exists $E_S\subset E$ 
such that \begin{equation*}\dim_{\rm H}E_S=\dim_{\rm H}E=\frac{1}{2}\ \text{ and }\ \overline{d}_q\left(\bigcup_{n\in\mathbb N}\bigcap_{x\in E_S}\{A_n(x)\}\cap S\right)=\overline{d}_q(S),\end{equation*} and 
\begin{equation*}\overline{d}\left(\bigcup_{n\in\mathbb N}\bigcap_{x\in E_{S}}\{\deg A_n(x)\}\cap \deg(S)\right)=\overline{d}(\deg(S)).\end{equation*}
In particular, 
\begin{equation*}\dim_{\rm H}\left\{
\begin{tabular}{l}
\!\!\!$x\in E\colon$\\
\!\!\!$\overline{d}_q(\{A_n(x)\colon n\in\mathbb N\}\cap S)=\overline{d}_q(S),$
\!\!\!\\
\!\!\!$\overline{d}(\{\deg A_n(x)\colon n\in\mathbb N\}\cap \deg(S) )$\!\!\!\\
\!\!\!$=\overline{d}(\deg(S))$\!\!\!
\end{tabular}
\right\}=\dim_{\rm H}E.\end{equation*}

\end{cor}
It provides a
direct mechanism for transferring combinatorial statements about subsets of $\mathbb N$ with positive upper density to statements about degree sets arising from
continued fraction expansions of Laurent series on sets of Hausdorff dimension
$1/2$.
For instance, combining Corollary~\ref{cor-FS} with Szemer\'edi's theorem \cite{S75}, we obtain the following: if
$S\subset \mathbb F_q^1[X]$ has positive upper density and $\deg(S)\subset \mathbb N$ also has positive upper density, then there exists
a subset of $E$ of Hausdorff dimension $1/2$ such that
for each $x$ in the subset, the degree set arising from its continued fraction expansion contains arithmetic progressions of arbitrary length inside $\deg(S)$.

We proceed to an important example of uniform transfer via Main Theorem.  Let
\[\mathcal{I}_q:=\{P\in\mathbb{F}_q^1[X]\colon P \text{ is irreducible}\}.\]
L\^e \cite{Le} established a function-field analogue of the Green--Tao theorem \cite{GT08},
showing that every subset of $\mathcal{I}_q$ of positive relative upper density contains
affine configurations of arbitrarily large dimension. More precisely, the following holds.

\begin{thm}[{\cite[Theorem~2]{Le}}]\label{BL-thm} 
Suppose that $U\subset \mathcal{I}_q$ satisfies
$\overline{d}_q(U| \mathcal{I}_q)>0.$
Then for every $k\in \mathbb{N}$, there exist $F,G\in \mathbb{F}_q[X]$ with $G\neq 0$ such that
\[
F+AG\in U
\qquad
\text{for every } A\in \mathbb{F}_q[X] \text{ with } \deg A<k.
\]

\end{thm}
For related results, see \cite{BLM}.
Since $\mathcal I_q$ has growth density with exponent $1$ (see \cite{LN, We}), by transferring Theorem~\ref{BL-thm} and Szemer\'edi's theorem \cite{S75} via Main Theorem we obtain the following theorem. 
\begin{thm}\label{cor-N}

Let $U \subset \mathcal{I}_q$ satisfy $\overline{d}_q(U | \mathcal{I}_q) > 0$ and $\overline{d}(\deg(U))>0$. Then there exists a set
$E_{\mathcal{I}_q, U} \subset E$ such that
\[\dim_{\rm H} E_{\mathcal{I}_q, U} = \frac{1}{2}\] and the following holds:
\begin{itemize}
\item for every $x \in E_{\mathcal{I}_q, U}$ and every $k \in \mathbb{N}$, the set
$\{A_n(x)\colon n \in \mathbb{N}\} \cap U$
contains an affine configuration of the form
\[\{F + AG \colon A \in \mathbb{F}_q[X],\ \deg A < k\}\]
for some $F,G \in \mathbb{F}_q[X]$ with $G \neq 0$;  
\item for every
$x \in E_{\mathcal{I}_q, U}$, the degree set
\[\{\deg A_n(x) \colon n \in \mathbb{N}\} \cap \deg(U)\]
contains arithmetic progressions of arbitrary length.
\end{itemize}
\end{thm}

\subsection{Structure of the paper}
The rest of the paper is organized as follows. 
In Section~2, we collect several preliminary results needed in the sequel. 
In Section~3, we follow the general strategy of \cite[Section 3]{NT}; however, several substantial modifications are required in the present setting. In particular, both the construction of the Cantor-like set, called the extreme seed set, and the insertion arguments must be reformulated so as to accommodate continued fraction expansions of Laurent series. 
In Section~4, we apply the construction developed in Section~3 to complete the proof of Main Theorem. Unlike the situation in \cite{NT}, our argument must simultaneously deal with two different notions of density, one on subsets of $\mathbb{F}_q^1[X]$ and the other on the corresponding degree sets in $\mathbb{N}$. For this reason, the construction has to be carried out carefully so that both densities are preserved at the same time.

\section{Preliminaries}
This section summarizes preliminary results needed for the proofs.

\subsection{Fundamental cylinders for continued fractions}\label{CF-sec}

For $n\in\mathbb N$ and
$(A_1,\ldots, A_n)\in (\mathbb F_q^1[X])^n$, we define an {\it $n$-th fundamental cylinder} by
\[I(A_1,\ldots, A_n)=\{x\in \Delta\colon A_i(x)=A_i\ \text{ for }i=1,\ldots,n\},\]
where each $A_i(x)$ is defined by \eqref{LSCF}.
%Fundamental cylinders describe the basic pieces of $\Delta$ determined by
%finite initial blocks of digits.
Let $|\cdot|$ denote the diameter of sets in $\Delta$. 
The following lemma gives the
exact diameter of a fundamental cylinder.
\begin{lem}[\cite{HH, Ni, Wu}]
\label{proper}
For every $n\in \mathbb N$ and every $(A_1,\ldots, A_n)\in (\mathbb F_q^1[X])^n$ we have $I(A_1,\ldots, A_n)$ is a disc with diameter
\[|I(A_1,\ldots, A_n)|=q^{-2\sum_{i=1}^n {\rm deg}A_i-1}.\]
\end{lem}
Besides this metric description, we will also use the simple but important
topological consequences of the non-Archimedean structure. We record them for later reference.
\begin{rem}
\label{remnon}
Since $\|\cdot\|_q$ is non-Archimedean, that is,
\[\|A+B\|_q \le \max\{\|A\|_q,\|B\|_q\},\]
every disc in $\Delta$ is both open and closed (clopen). In particular, every fundamental cylinder is clopen. Moreover, any two discs in $\Delta$ that intersect are nested: if $D_1\cap D_2\neq\emptyset$, then either $D_1\subset D_2$ or $D_2\subset D_1$.
\end{rem}

 \subsection{Lower bounds of Hausdorff dimension}\label{dimension-sec}
 For $x\in \Delta$ and $r>0,$ let $B(x, r)=\{y\in \Delta\colon \|x-y\|_q<r\}.$ 
 We first recall a standard method for deriving dimension lower bounds from measures.
\begin{lem}\textnormal{(\cite[Lemma~3.2]{HH}; see also \cite{Fal97, Fal14})}
\label{mass1}
Let $\Lambda$ be a Borel subset of $\Delta$ and let $\nu$ be a Borel probability measure on $\Delta$ with $\nu(\Lambda)=1$. If there exists $\lambda>0$ such that for any $x\in \Lambda$,
\[\liminf_{r\to0}\frac{\log\nu(B(x,r))}{\log r}\geq \lambda,\] then $\dim_{\rm H}\Lambda\geq \lambda$.
\end{lem}

In order to compare the dimensions of two sets through maps between them,
we recall some notions.
Let $F\subset \Delta$ be a set. We say $f\colon F\to \Delta$ is {\it H\"older continuous with exponent} $\gamma\in(0,1]$ if there exists $C>0$ such that
\[||f(x)-f(y)||_{q}\le C||x-y||_q^{\gamma}\ \text{ for all } x, y\in F.\]
The following lemma shows that \cite[Proposition~3.3]{Fal14}, originally proved in the Euclidean setting, remains valid in the setting of $\Delta$.
\begin{lem}
\label{Holder-F}
Let $F\subset \Delta$ and let $f\colon F\to \Delta$ be H\"older continuous with exponent $\gamma\in(0,1]$.
Then \[\dim_{\rm H}F\geq\gamma\cdot\dim_{\rm H}f(F).\]\end{lem}
We say $f\colon F\to \Delta$ is {\it almost Lipschitz} if
for any $\gamma\in(0,1)$,
$f$ is H\"older continuous with exponent $\gamma$.
Applying Lemma~\ref{Holder-F} and letting $\gamma\to 1$, we obtain the following.
\begin{lem}\label{Holder}Let $F\subset \Delta$ and let $f\colon F\to \Delta$ be almost Lipschitz. Then \[\dim_{\rm H} F\geq\dim_{\rm H}f(F).\]
\end{lem}

\subsection{Convergence exponent}\label{conv-sec}
Since we study continued fractions whose digits are restricted to a prescribed subset of $\mathbb F_q^1[X]$, we need a dimension result for such restricted digit sets. For this purpose, we recall the notion of convergence exponent.
Let $S\subset \mathbb F_q^1[X].$
 Define
\[\tau(S)=\inf\left\{s\geq0\colon\sum_{ A\in S}(q^{-2{\rm deg}A})^s<\infty\right\}.\]
We call $\tau(S)$ a {\it convergence exponent} of $S$. The following theorem treats the Hausdorff dimension of the set of points
whose digits all lie in $S$ and whose degrees tend to infinity.

 \begin{lem}[{\cite[Theorem~2.5]{HW}}]\label{HirstLau} For any infinite subset $S$ of $\mathbb F_q^1[X]$, we have \[\dim_{\rm H}\{x\in \Delta \setminus \mathbb F_q(X)\colon \{A_n(x)\colon n\in\mathbb N\}\subset S\text{ and }{\rm deg}A_n(x)\to\infty\text{ as }n\to\infty\}=\tau(S).\]\end{lem}
To estimate $\tau(S)$, it is convenient to express it in terms of the number of elements of $S$ at each degree level.
\begin{lem}[{\cite[Lemma~3.6]{HW}}]\label{expconv}
Let $S$ be an infinite subset of $\mathbb F_q^1[X]$. For any $n\ge 1,$ let $D_n={\rm card}\{A\in S\colon {\rm deg}A=n\}.$ Then \[\tau(S)=\limsup_{n\to \infty}\frac{\log D_n}{2n\log q}.\]
\end{lem}
We next show that if $S$ has growth density, then its convergence exponent is bounded above by the corresponding critical value.

\begin{lem}\label{disc2}If $S\subset\mathbb F_q^1[X]$ has growth density with exponent $\alpha\geq1$, then we have
\[\tau(S)\le \frac{1}{2\alpha}.\]
\end{lem}

\begin{proof}

Since $D_n \le \#Q_n(S)$ for any $n$ and $S$ has growth density with exponent $\alpha\geq1$, there exist $\gamma_2$ and $\beta\ge 0$ such that for all sufficiently large $n,$ 
\[D_n \le \gamma_2 \frac{q^{n/\alpha}}{n^\beta}.\]
Hence, for all sufficiently large $n,$ 
\[\frac{\log D_n}{2n \log q}\le 
\frac{1}{2\alpha}
+ \frac{\log \gamma_2 - \beta \log n}{2n \log q}.\]
Letting $n \to \infty$, by Lemma~\ref{expconv} we obtain $\tau(S)\le 1/(2\alpha).$

\end{proof}
Combining Lemmas \ref{HirstLau} and \ref{disc2} we obtain the following.
\begin{pro}
\label{dim_Hirst}
If $S\subset\mathbb F_q^1[X]$ has growth density with exponent $\alpha\geq1$, then we have
\[\dim_{\rm H}\{x\in \Delta\setminus \mathbb F_q(X)\colon \{A_n(x)\colon n\in\mathbb N\}\subset S\text{ and }{\rm deg}A_n(x)\to\infty\text{ as }n\to\infty\}\leq \frac{1}{2\alpha}.\]
\end{pro}

\section{Construction of extreme seed sets and insertion arguments}
This section is devoted to the main construction. 
We begin by constructing an {\it extreme seed set}, a Cantor-like set contained in $E$ whose digit sets are contained in a suitably chosen subset of the ambient set. 
We then establish an insertion procedure that allows us to modify this seed set in a controlled way while keeping the Hausdorff dimension unchanged.

\subsection{Extreme seed sets and dimension estimate}\label{C-const}
Let $S$ be an infinite subset of $\mathbb F_q^1[X]$.
For an infinite subset $K$ of $S$ and a natural number $t\ge 3$, define
\begin{equation}\label{seed-def}R_{t}(K)=\{x\in E\colon A_n(x)\in K \ \text{and}\ (2n)^t\le {\rm deg}A_n(x) < (2n+1)^t
\ \text{ for every}\ n \ge 1\}.\end{equation}
Non-empty sets of this form are called {\it seed sets}. %Seed sets are Moran fractals. 
 We say the seed set $R_{t}(K)$ is an {\it extreme seed set associated with $S$} if the following conditions hold:
 \begin{itemize}
 \item[(A1)] $S$ has growth density with exponent $\alpha\geq1$;
 \item[(A2)]
 $\overline{d}_q(K|S)=0;$ 
 \item[(A3)] $\dim_{\rm H}R_{t}(K)=1/(2\alpha)$.
 \end{itemize}
The aim of this section is to prove the following.

\begin{pro}
\label{seed-Prop}
 If $S\subset \mathbb F_q^1[X]$ has growth density with exponent $\alpha\geq1$, then there exist $S_{\ast}\subset S$ and a natural number $t\ge 3$ such that 
$R_{t}(S_{\ast})$ is an extreme seed set associated with $S$.
\end{pro}
To prove Proposition~\ref{seed-Prop}, we first introduce an auxiliary function that will be
used to quantify the sparsity of the subset selected from $S$. Define a function $\mu\colon [1,\infty)\to\mathbb N$ by \begin{equation}\label{mu-def}\mu(x)=k\ \text{ for } x\in[k!,(k+1)!),\ k\in\mathbb N.\end{equation} The following lemma controls its growth on exponential scales.

\begin{lem}[{\cite[Lemma~3.3]{NT}}]\label{simple}For any $b>e$ there exists $\sigma>1$ such that \[\mu(b^n)\leq (n\log b)^{\sigma}\ \text{ for every }n\in\mathbb N.\]\end{lem}

We now use the function $\mu$ to extract a subset $S_{\ast}\subset S$ which is sparse
enough to have relative density zero in $S$, but still sufficiently large for the
subsequent seed set construction.

\begin{lem}\label{find-C}Let
$S$ be an infinite subset of $\mathbb F_q^1[X]$ and enumerate $S$ without repetition as $S=\{A_n\in\mathbb F_q^1[X]\colon n\in\mathbb N\}$,  ${\rm deg}A_1\leq {\rm deg}A_2\leq\cdots.$
There exists a subset $S_*$ of $S$ such that 
  for all sufficiently large integers $T\geq1$,  \begin{equation}\label{seed-eq3}\frac{1}{2\mu(T)}\leq \frac{\#\{1\leq n\leq T\colon A_n\in S_* \}}{T }\leq \frac{3}{\mu(T)},\end{equation} 
  and \[\overline{d}_q(S_*|S)=0.\]
\end{lem}

\begin{proof}
Let \[P=\bigcup_{k=2}^\infty\{k!+ik\colon i=0,\ldots, k!-1\}\subset\mathbb N.\]
Although the following estimate is an immediate consequence of the proof of \cite[Lemma~3.2]{NT}, we include the details for the reader's convenience.

For every sufficiently large $T\geq 1$, we have
\begin{equation}\label{eq-lem33-1}
\frac{T}{2\mu(T)}\leq \#(P\cap[1,T])\leq \frac{3T}{\mu(T)}.
\end{equation}
Indeed, for every $T\geq 1$ let $k\geq 1$ be the integer such that $T\in [k!,(k+1)!)$. Then for all sufficiently large $k\ge 1$
\[\begin{split}
\#(P\cap[1,T])
&= \#(P\cap[1,(k-1)!))+\#(P\cap[(k-1)!,k!))+\#(P\cap[k!,T]) \\
&\leq \#(P\cap[1,(k-1)!))
   +\frac{\#([(k-1)!,k!)\cap\mathbb N)}{k-1}
   +\frac{\#([k!,T]\cap\mathbb N)}{k}+1 \\
&\leq (k-1)!+(k-1)!+\frac{\#([k!,T]\cap\mathbb N)}{k} \\
&\leq \frac{3T}{\mu(T)},
\end{split}\]
and
\[\begin{split}
\#(P\cap[1,T])
&\geq \sum_{i=2}^{k-1}\frac{\#([i!,(i+1)!)\cap\mathbb N)}{i}
   +\frac{\#([k!,T]\cap\mathbb N)}{k}-1 \\
&\geq \sum_{i=2}^{k-1}\frac{\#([i!,(i+1)!)\cap\mathbb N)}{k}
   +\frac{\#([k!,T]\cap\mathbb N)}{k}-1 \\
&= \frac{\#([2,T]\cap\mathbb N)}{k}-1 \\
&\geq \frac{T}{2\mu(T)},
\end{split}\]
which implies \eqref{eq-lem33-1}.
Set
\[
S_{\ast}=\{A_n\in S\colon n\in P\}.
\]
Then for every $T\geq 1$,
\[
\#\{1\leq n\leq T\colon A_n\in S_{\ast}\}=\#(P\cap[1,T]).
\]
Hence, \eqref{seed-eq3} follows from \eqref{eq-lem33-1}.

For each integer $N\geq1$ put
$T_N=\max\{n\geq1\colon {\rm deg}A_n\leq N\}$ for $N$ satisfying $Q_N(S)\neq \emptyset$.
Since ${\rm deg}A_n\to \infty$ monotonically, we have $n\leq T_N$ implies ${\rm deg}A_n\leq N.$ By the maximality of  $T_N$, we have  $n> T_N$ implies ${\rm deg}A_n> N.$ Combining this we have $Q_N(S)=\{A_1,..., A_{T_N}\}$ and hence\begin{equation}\label{claim1}\#Q_N(S)=T_N .\end{equation}
Then \eqref{claim1} implies 
\begin{equation}\label{claim2}\#Q_N(S_*) =\#\{1\leq n\leq
T_N\colon A_n\in S_* \}.\end{equation}
By \eqref{claim1}, \eqref{claim2} and the upper bound in \eqref{seed-eq3}, for all sufficiently large $N$ we have
\[\frac{\#Q_N(S_*) }{\#Q_N(S) }=\frac{\#\{1\leq n\leq
T_N\colon A_n\in S_* \}}{T_N }
\leq\frac{3}{\mu(T_N )}.\]
Letting $N\to\infty$ yields $\overline{d}_q(S_*|S)=0$ as required. \end{proof}

Lemma~\ref{find-C} provides a subset $S_{\ast}\subset S$ which is sparse enough to satisfy (A2), yet still rich enough to leave many admissible digits in the degree windows used for the seed set construction. Let $t\ge 3$ be a natural number. For $n\ge 1$, let
\[\mathcal C_n=\{A\in S_{\ast}\colon (2n)^t\le \deg A < (2n+1)^t\}.\]
The next lemma shows that each $\mathcal C_n$ contains sufficiently many elements for all $n$ if we take $t$ sufficiently large.

\begin{lem}
\label{Cardi}
There exist a natural number $t\ge 3$ and constants $c_0>0$ and $\rho>0$ such that for all
integers $n\ge 1$,
\[
\# \mathcal C_n \ge c_0\frac{q^{(2n+1)^t/\alpha}}{n^{\rho}}.
\]
\end{lem}

\begin{proof}
By Lemma~\ref{find-C} for all sufficiently large $N$,
\begin{equation}\label{eq-lem34-1}
\frac{\#Q_N(S)}{2\mu(\#Q_N(S))} \le \#Q_N(S^{\ast}) \le \frac{3\#Q_N(S)}{\mu(\#Q_N(S))}.
\end{equation}

For each $n,$ by the definition of $\mathcal C_n$ we have
\[\#\mathcal C_n = \#Q_{(2n+1)^t-1}(S^{\ast}) - \#Q_{(2n)^t-1}(S^{\ast}).\]
By \eqref{eq-lem34-1} and taking $t$ sufficiently large, for all $n$ we obtain
\begin{equation}\label{eq-lemma34-2}
\#\mathcal C_n\ge\frac{\#Q_{(2n+1)^t-1}(S)}{2\mu(\#Q_{(2n+1)^t-1}(S))}-\frac{3\#Q_{(2n)^t-1}(S)}{\mu(\#Q_{(2n)^t-1}(S))}.\end{equation}
Since $S$ has growth density with exponent $\alpha$, there exist constants
$\gamma_1,\gamma_2>0$ and $\beta\ge 0$ such that for all sufficiently large $N$,
\[\gamma_1\frac{q^{N/\alpha}}{N^\beta}\le \#Q_{N}(S) \le\gamma_2\frac{q^{N/\alpha}}{N^\beta}.\]
Hence, taking $t$ sufficiently large again, there exist some constants $c_1,c_2>0$ such that for all $n$,
\begin{equation}\label{eq-lemma34-3}
\#Q_{(2n+1)^t-1}(S)\ge c_1 \frac{q^{(2n+1)^t/\alpha}}{n^{\beta t}},\
\#Q_{(2n)^t-1}(S)\le c_2 \frac{q^{(2n)^t/\alpha}}{n^{\beta t}}.
\end{equation}

Since
$\#Q_{(2n+1)^t-1}(S)\leq q^{(2n+1)^t}$, by Lemma~\ref{simple} it follows that there exist $c_3>0$ and $\sigma>0$ such that for all $n$,
\begin{equation}\label{eq-lem34-4}
\mu(\#Q_{(2n+1)^t-1}(S))\leq   c_3 n^{t\sigma}.
\end{equation}

Combining \eqref{eq-lemma34-2}, \eqref{eq-lemma34-3}, and
\eqref{eq-lem34-4}, there exists $t\ge 3$ such that for all
 $n$ we have
\[\#\mathcal C_n\geq 
 c_1c_3^{-1} \frac{q^{(2n+1)^t/\alpha}}{n^{(\beta+\sigma)t}}-
3c_2 \frac{q^{(2n)^t/\alpha}}{n^{\beta t}}.\]
Hence, there exists $c_0>0$ such that for all $n$ we have 
\[\#\mathcal C_n \geq c_0\frac{q^{(2n+1)^t/\alpha}}{n^{(\beta+\sigma)t}},\]
which completes the proof.
\end{proof}

%Lemma~\ref{Cardi} shows that each degree window contains exponentially many admissible digits, up to a polynomial loss. 
We now set
\begin{equation}
\label{Fset}F=\bigcap_{n= 1}^\infty\bigcup_{\substack{ A_i\in \mathcal C_i\\ 1\leq i\leq n}}{I(A_{1},\ldots, A_n)}.\end{equation}
Here, choose $t$ sufficiently large so that the conclusion of Lemma~\ref{Cardi} holds.
The next lemma equips
this set with a natural probability measure and yields the required lower bound
for its Hausdorff dimension.
\begin{lem}\label{gibbsnew}

There exists a Borel probability measure $\nu$ supported on $F$ such that for all $(A_n)_{n=1}^{\infty}\in \prod_{n=1}^{\infty}\mathcal C_n$ we have
\[\liminf_{n\to \infty}\frac{\log \nu(I(A_1,..., A_n))}{\log |I(A_1,..., A_n)|}\ge \frac{1}{2\alpha}.\]

\end{lem}
\begin{proof}
For each $n\in\mathbb N$ let $P_n$
be a measure on $\Delta$ such that \[P_n(I(A_1,..., A_n))=\frac{1}{\# \mathcal C_1\cdots \#\mathcal C_n} \ \text{for}\ (A_1,..., A_n)\in \mathcal C_1\times\cdots \times \mathcal C_n.\]
By Lemma~\ref{Cardi}
there exist $c_0>0$ and $\rho>0$ such that for all  $n\ge 1$,
\[
\# \mathcal C_n \ge c_0\frac{q^{(2n+1)^t/\alpha}}{n^{\rho}}.
\]
Hence, for all $n\ge 1$, for each  $(A_1,..., A_n)\in \mathcal C_1\times\cdots \times \mathcal C_n$ we have 
\[P_n(I(A_1,..., A_n))\leq \frac{(n!)^{\rho}}{c_0^n q^{\alpha^{-1}\sum_{j=1}^n(2j+1)^t }}.\]
By Kolmogorov's extension theorem,
there is a Borel probability measure $\nu$ supported on $F$ such that for all $n\in\mathbb N$ and all $(A_1,..., A_n)\in \mathcal C_1\times\cdots \times \mathcal C_n$, 
\begin{equation}\label{mu-ineq}\nu(I(A_1,..., A_n))\leq \frac{(n!)^{\rho}}{c_0^n q^{\alpha^{-1}\sum_{j=1}^n(2j+1)^t }}.\end{equation}
For $1\leq j\leq n$, 
the definition of $\mathcal C_j$ gives $\deg A_j<(2j+1)^t$.  By this and Lemma~\ref{proper}, 
\begin{equation}\label{mu-ineq3}|I(A_1,\ldots, A_n)|=q^{-2\sum_{j=1}^n {\rm deg}A_j-1}>q^{-2\sum_{j=1}^n (2j+1)^t-1}.\end{equation}
Combining \eqref{mu-ineq} and \eqref{mu-ineq3}, for all $n\ge 1$ we have
\begin{equation*}
\begin{split}\frac{\log \nu(I(A_1,..., A_n))}{\log |I(A_1,..., A_n)|}\ge 
\frac{n\log c_0+\alpha^{-1}\log q\sum_{j=1}^{n}(2j+1)^t-\rho \sum_{j=1}^n \log j}{2\log q\sum_{j=1}^{n}(2j+1)^t+\log q},
\end{split}
\end{equation*}
which implies \[\liminf_{n\to \infty}\frac{\log \nu(I(A_1,..., A_n))}{\log |I(A_1,..., A_n)|}\ge \frac{1}{2\alpha},\]as required.

\end{proof}
We are now in a position to prove Proposition~\ref{seed-Prop}.
\begin{proof}[Proof of Proposition~\ref{seed-Prop}]
Let $S_{\ast}$ be as in Lemma~\ref{find-C}, choose $t$ sufficiently large so that the conclusion of Lemma~\ref{Cardi} holds and let $F$ be the set defined by (\ref{Fset}). 
By construction of $F$ we have $F\subset R_t(S_{\ast})$.
Since $S$ has growth density with exponent $\alpha\geq1$ and $\overline{d}_q(S_{\ast}|S)=0$ we verify (A1) and (A2). Take $x\in F$ and $0<r<1$. Choose $n$ such that
\[|I(A_1(x),\dots,A_n(x), A_{n+1}(x))|\le r < |I(A_1(x),\dots,A_n(x))|.\]
Since $I(A_1(x),\dots,A_n(x))$ is a disc and $x\in I(A_1(x),\dots,A_n(x))$, by Remark~\ref{remnon}
\[B(x,r)\subset I(A_1(x),\dots,A_n(x)).\]
Let $\nu$ be the Borel probability measure from Lemma~\ref{gibbsnew}.
Then
\[\frac{\log \nu(B(x,r))}{\log r}\ge \frac{\log \nu (I(A_1(x),\dots,A_n(x)))}{\log |I(A_1(x),\dots,A_n(x), A_{n+1}(x))|}.\]
Since $(2n)^t\le {\rm deg}A_n(x) < (2n+1)^t$
for every\ $n \ge 1$, by Lemma~\ref{proper} we have \[\frac{\log |I(A_1(x),\dots,A_n(x), A_{n+1}(x))|}{\log |I(A_1(x),\dots,A_n(x))|}\to 1 \ \text{as} \ n\to \infty.\]
Combining this with Lemma~\ref{gibbsnew} yields
\[\liminf_{r\to 0}\frac{\log \nu(B(x,r))}{\log r}\ge \liminf_{n\to \infty}\frac{\log \nu (I(A_1(x),\dots,A_n(x)))}{\log |I(A_1(x),\dots,A_n(x), A_{n+1}(x))|}\ge 1/(2\alpha).\]
By Lemma~\ref{mass1}, this implies $\dim_{\rm H}R_t(S_{\ast})\ge \dim_{\rm H} F \ge 1/(2\alpha)$.

On the other hand,
\[R_t(S^*)\subset
\{x\in E \colon A_n(x)\in S\text{ for all }n\ge 1 \text{ and } \deg A_n(x)\to\infty\}.\]
By Proposition~\ref{dim_Hirst} we obtain 
$\dim_{\rm H} R_t(S_{\ast})\le {1}/({2\alpha}).$
Hence,
\[\dim_{\rm H} R_t(S_{\ast})=\frac{1}{2\alpha},\] which verifies (A3).
This completes the proof.

\end{proof}

\subsection{Insertion arguments}\label{insert-sec}
We next move to insertion arguments.
Let $V$, $Z$ be infinite subsets of $\mathbb F_q^{1}[X]$.
Let $t\ge 3$ be a natural number and suppose $R_{t}(Z)\neq\emptyset$. 
We add elements of $V$ into the digit sequences of points in the seed set $R_{t}(Z)$, and construct a new subset of $\Delta$ in the following manner.

Let $M_0=0$, and let $(M_k)_{k=1}^\infty$ be a strictly increasing sequence of positive integers. Let $(W_k)_{k=1}^\infty$ be a sequence of finite subsets of $V$. Write \[W_k=\{w^{(k)}_i\colon i=1,\ldots,i(k)\},\ \deg w^{(k)}_1\le \cdots \le \deg w^{(k)}_{i(k)}.\]
Let $y\in R_{t}(Z).$
For each $k\geq1$ with $W_k\neq\emptyset$, we
 add the elements of $W_{k}$ into the digit sequence $(A_n(y))_{n=1}^\infty$ to define a new sequence 
 \[\begin{split}&\ldots,A_{M_{k}-1}(y),\ A_{M_{k}}(y),\ w_{1}^{(k)},\ldots,w_{i(k)}^{(k)},\ A_{M_{k}+1}(y),\ldots\\
&\ldots,A_{M_{k+1 }}(y),\ w_{1}^{(k+1)},\ldots,w_{i(k+1)}^{(k+1)},\ A_{M_{k+1 }+1}(y),\ldots\end{split}\]
Let $x(y)$ denote the point in $\Delta$ whose continued fraction expansion \eqref{LSCF2} is given by this new sequence. Let $R_{t}(Z,V,(M_k)_{k=1}^\infty,(W_k)_{k=1}^\infty)$ denote the collection of these points:
\[R_{t}(Z,V,(M_k)_{k=1}^\infty,(W_k)_{k=1}^\infty)=\{x(y)\in \Delta\colon y\in R_{t}(Z)\}.\] 
The map $y\in R_{t}(Z)\mapsto x(y)\in R_{t}(Z,V,(M_k)_{k=1}^\infty,(W_k)_{k=1}^\infty)$ is bijective.
Let \[f_{t,Z,V,(M_k)_{k=1}^\infty,(W_k)_{k=1}^\infty}\colon  R_{t}(Z,V,(M_k)_{k=1}^\infty,(W_k)_{k=1}^\infty)\to R_{t}(Z)\] denote the inverse of this map. We call $f_{t,Z,V,(M_k)_{k=1}^\infty,(W_k)_{k=1}^\infty}$ an {\it elimination map}. 
Clearly, if $V\cap Z=\emptyset$ and the elements of $(W_k)_{k=1}^\infty$ are pairwise disjoint then $R_{t}(Z,V,(M_k)_{k=1}^\infty,(W_k)_{k=1}^\infty)$ is contained in $E$. Moreover, the elimination map $f$ sends $x\in R_{t}(Z,V,(M_k)_{k=1}^\infty,(W_k)_{k=1}^\infty)$ to $f(x)\in R_{t}(Z)$ whose digit sequence is given by  eliminating elements of $\bigcup_{k=1}^\infty W_k$ from the digit sequence of $x$. 

To preserve the Hausdorff dimension under the insertion procedure, it is enough
to show that the associated elimination map is almost Lipschitz. The next
proposition gives a sufficient condition for this.

\begin{pro}\label{pro-almLip}
Let $(\varepsilon_k)_{k=1}^\infty$ be a decreasing sequence of positive reals converging to $0$.
Let $V$, $Z$ be infinite subsets of $\mathbb F_q^1[X]$.
Let $t\ge 3$ be a natural number and suppose $R_{t}(Z)\neq\emptyset$. Let $M_0=0$ and let $(M_k)_{k=1}^\infty$ be a strictly increasing sequence of positive integers, and
let $(W_k)_{k=1}^\infty$ be a sequence of finite subsets of 
$V$ such that for every $k\geq1$, 
\begin{equation}\label{condi-Lip-1}\sum_{\omega\in W_k}\deg \omega\le \varepsilon_k \sum_{i=M_{k-1}+1}^{M_k}(2i)^t\ \text{ if }W_k\neq\emptyset\end{equation}
and
\begin{equation}\label{condi-Lip-2} 
\sum_{i=1}^{n+1}(2i+1)^t\le (1+\varepsilon_k)\sum_{i=1}^n(2i)^t\ \text{ for every }n\geq M_k.\end{equation}
Then the elimination map $f=f_{t,Z,V,(M_k)_{k=1}^\infty,(W_k)_{k=1}^\infty}$ is almost Lipschitz.
 \end{pro}

\begin{proof}

Let us abbreviate $R_{t}(Z)$,
 $R_{t}(Z,V,(M_k)_{k=1}^\infty,(W_k)_{k=1}^\infty)$
 to  $R(Z)$,
 $R(Z,V)$ respectively.
Let $k\in\mathbb N$.
For a pair $y_1,y_2$ of distinct points in $R(Z)$, let
\[s(y_1,y_2)=\min \{n\ge 0\colon  A_{n+1}(y_1)\neq A_{n+1}(y_2)\}.\]
For each $j\ge 1$ set $m_j=\# W_j.$
If $s(y_1,y_2)<M_k$, then the points $x_i:=f^{-1}(y_i)$ differ within the first
$M_k+\sum_{j=1}^k m_j$ digits. Since only finitely many words of length $M_k+\sum_{j=1}^k m_j$ can occur as initial blocks of points in $R(Z,V)$, there exists $\delta_k>0$ such that if $s(y_1,y_2)<M_k$, then
$\|x_1-x_2\|_q \ge \delta_k.$
 Since $\|y_1-y_2\|_q\le 1$, it follows that
\begin{equation}\label{eq-pro36-1}
\|y_1-y_2\|_q \le \delta_k^{-1}\|x_1-x_2\|_q\ \text{if } s(y_1,y_2)<M_k.\end{equation}

Let $x_1,x_2\in R(Z,V)$ be distinct and put $y_1:=f(x_1),\ y_2:=f(x_2).$
Assume that $n:=s(y_1,y_2)\ge M_k.$
Let $\ell \ge k$ be the unique integer such that
$M_\ell\le n < M_{\ell+1}.$
Then $x_1$ and $x_2$ have the same first $n+m_1+\cdots+m_\ell$
digits, and their next digits are $A_{n+1}(y_1)$ and $A_{n+1}(y_2)$, respectively.
Hence, \[x_1\in I(A_1(x_1),\ldots, A_{n+m_1+\cdots +m_\ell}(x_1), A_{n+1}(y_1)),\] and \[x_2\notin I(A_1(x_1),\ldots, A_{n+m_1+\cdots +m_\ell}(x_1), A_{n+1}(y_1)).\]
By this, Lemma~\ref{proper}, and Remark~\ref{remnon}, we obtain
\[\|x_1-x_2\|_q\ge |I(A_1(x_1),\ldots, A_{n+m_1+\cdots +m_\ell}(x_1), A_{n+1}(y_1))|.\]
Applying Lemma~2.1 again, we have
\[\|x_1-x_2\|_q\ge
q^{-1}q^{-2\left(\sum_{i=1}^{n+1}\deg A_i(y_1)+\sum_{j=1}^\ell\sum_{w\in W_j}\deg w\right)}.\]
Since $\deg A_i(y_1)<(2i+1)^t$ for every $i$, we have
\begin{equation}\label{eq-pro36-2}
\|x_1-x_2\|_q
\ge
q^{-1}q^{-2 \sum_{i=1}^{n+1}(2i+1)^t}q^{-2\sum_{j=1}^\ell\sum_{w\in W_j}\deg w}.\end{equation}
By \eqref{condi-Lip-1} and the monotonicity of $(\varepsilon_j)_{j=1}^{\infty}$, we have
\[
\sum_{j=1}^\ell\sum_{w\in W_j}\deg w
\le
\sum_{j=1}^k \sum_{\omega\in W_j}\deg \omega+
\varepsilon_k \sum_{j=k+1}^\ell \sum_{i=M_{j-1}+1}^{M_j} (2i)^t
\le
\sum_{j=1}^k \sum_{\omega\in W_j}\deg \omega+
\varepsilon_k \sum_{i=1}^{n} (2i)^t.
\]
Combining this with \eqref{eq-pro36-2}, we have
\[\|x_1-x_2\|_q
\ge q^{-1}q^{-2 \sum_{i=1}^{n+1}(2i+1)^t}q^{-2\sum_{j=1}^k \sum_{\omega\in W_j}\deg \omega}q^{-2\varepsilon_k \sum_{i=1}^{n} (2i)^t}.\]
The condition \eqref{condi-Lip-2} yields
\begin{equation}\label{eq-pro36-3}
\|x_1-x_2\|_q \ge q^{-1-2\sum_{j=1}^k \sum_{\omega\in W_j}\deg \omega}q^{-2(1+2\varepsilon_k)\sum_{i=1}^n(2i)^t}.
\end{equation}
On the other hand, since $y_1, y_2\in I\bigl(A_1(y_1),\dots,A_n(y_1)\bigr),$ Lemma~\ref{proper} and the lower bounds $\deg A_i(y_1)\ge (2i)^t$ imply
\begin{equation}\label{eq-pro36-4}
\|y_1-y_2\|_q\le\left|I\bigl(A_1(y_1),\dots,A_n(y_1)\bigr)\right|
=q^{-1}q^{-2\sum_{i=1}^n \deg A_i(y_1)}
\le
q^{-1}q^{-2\sum_{i=1}^n (2i)^t}.
\end{equation}
By \eqref{eq-pro36-3} and \eqref{eq-pro36-4},
\begin{equation}
\begin{split}
\|x_1-x_2\|_q
&\ge
q^{-1-2\sum_{j=1}^k \sum_{\omega\in W_j}\deg \omega}q^{-2(1+2\varepsilon_k)\sum_{i=1}^n(2i)^t}\\
&=
q^{2\varepsilon_k-2\sum_{j=1}^k \sum_{\omega\in W_j}\deg \omega}(q^{-1}q^{-2\sum_{i=1}^n(2i)^t})^{1+2\varepsilon_k}\\
&\ge q^{2\varepsilon_k-2\sum_{j=1}^k \sum_{\omega\in W_j}\deg \omega}\|y_1-y_2\|_q^{1+2\varepsilon_k}.
\end{split}
\end{equation}
Hence,
\[\|f(x_1)-f(x_2)\|_q=\|y_1-y_2\|_q
\le C_k \|x_1-x_2\|_q^{1/(1+2\varepsilon_k)},\]
where
\[
C_k:=q^{(2\sum_{j=1}^k \sum_{\omega\in W_j}\deg \omega-2\varepsilon_k)/(1+2\varepsilon_k)}.
\]
Combining this with \eqref{eq-pro36-1}, there exists a
constant $C_k^{\prime}>0$ such that
\begin{equation*}
\|f(x_1)-f(x_2)\|_q \le C_k^{\prime} \|x_1-x_2\|_q^{1/(1+2\varepsilon_k)}
\ \text{for all }\  x_1,x_2\in R(Z,V).
\end{equation*}
Letting $\varepsilon_k\to 0$, $f$ is almost Lipschitz, which is required.\end{proof}

\section{Proof of Main Theorem}\label{sec-proof}
Let $S\subset \mathbb F_q^1[X]$ have growth density with exponent $\alpha\ge 1$, and
let $U\subset S$ satisfy
$\overline{d}_q(U| S)>0.$
For each $d\in \deg(U),$ take an element $A_d\in U$ such that $d=\deg A_d$ and set 
\[U_{\deg}:=\{A_d\in U\colon d\in \deg(U)\}.\] Since $Q_N(U_{\deg})$ satisfies $\# Q_N(U_{\deg})\le N$ for any $N\ge 1,$ it follows that $\tilde{S}:=S\setminus U_{\deg}$ also has growth density with exponent $\alpha\ge 1.$

By Proposition~\ref{seed-Prop}, there exist a subset $S_{\ast}\subset \tilde{S}$ and
a natural number $t\ge 3$ such that $R_t(S_{\ast})$ is an extreme seed set associated with $\tilde{S}$. Since $S_{\ast}\subset \tilde S\subset S$ and 
$\overline{d}_q(S_{\ast}| \tilde S)=0$ we have $\overline{d}_q(S_{\ast}| S)=0$ and 

\begin{equation*}
\overline{d}_q(U\setminus S_{\ast} | S)=\overline{d}_q(U| S).
\end{equation*}

Let $(\varepsilon_k)_{k=1}^\infty$ be a decreasing sequence of positive real
numbers converging to $0$. Choose a strictly increasing sequence
$(N_k)_{k=1}^\infty$ of positive integers such that
\begin{equation}\label{eq-sec4-1}
\lim_{k\to\infty}
\frac{\#Q_{N_k}(U\setminus S_{\ast})}{\#Q_{N_k}(S)}=\overline{d}_q(U\setminus S_{\ast}| S)=\overline{d}_q(U| S).
\end{equation}
Set $N_0:=0$, and for each $k\ge 1$ define
\[W_k:=Q_{N_k}(U\setminus S_{\ast})\setminus Q_{N_{k-1}}(U\setminus S_{\ast}).\]
Then each $W_k$ is a finite subset of $U\setminus S_{\ast}$, the family
$\{W_k\colon k\ge 1\}$ is pairwise disjoint, and
\begin{equation}\label{eq-sec4-2}
\bigcup_{j=1}^k W_j = Q_{N_k}(U\setminus S_{\ast})
\ \text{for all}\ k\ge 1.
\end{equation}

We now construct a strictly increasing sequence $(M_k)_{k=1}^\infty$ inductively.
Set $M_0:=0$. Suppose that $M_0,\dots,M_{k-1}$ have already been chosen. Since
\[\sum_{i=M_{k-1}+1}^{\ell } (2i)^t\to \infty\ \text{as}\ \ell\to \infty,\]
we can choose $M_k>M_{k-1}$ such that
\begin{equation}\label{eq-sec4-3}
\sum_{\omega\in W_k}\deg \omega \le \varepsilon_k \sum_{i=M_{k-1}+1}^{M_k} (2i)^t.\end{equation}
Moreover, since
\[\frac{\sum_{i=1}^{n+1}(2i+1)^t}{\sum_{i=1}^{n}(2i)^t}\to 1 \ \text{as} \ n\to\infty,\]
we may also assume that
\begin{equation}\label{eq-sec4-4}
\sum_{i=1}^{n+1}(2i+1)^t\leq (1+\varepsilon_k)\sum_{i=1}^{n}(2i)^t\ \text{for every } n\ge M_k.\end{equation}
Define
\[E_{S,U}:=R_t\bigl(S_{\ast}, U\setminus S_{\ast},(M_k)_{k=1}^\infty,(W_k)_{k=1}^\infty\bigr).\]
By construction, every point of $E_{S,U}$ is obtained from a point of $R_t(S_{\ast})$ by inserting the same blocks $W_k$ at the position $M_k$.

\begin{cla*}
\begin{equation}\label{hougan}E_{S,U}\subset \{x\in E \colon A_n(x)\in S \text{ for all } n\ge 1\}.
\end{equation}
\end{cla*}
\begin{proof}[Proof of Claim]
Take $x\in E_{S,U}$, and let $y\in R_t(S_{\ast})$ be the corresponding point, namely $y=f(x)$, where $f$ is the elimination map. Since $y\in E$, its digits are pairwise distinct. Moreover, each inserted block $W_k$ is contained in $U\setminus S_{\ast}$
and the sets $W_k$ are pairwise disjoint. Hence, the inserted digits are pairwise distinct and are all disjoint from the seed digits from $S_{\ast}$. Hence, the
digits of $x$ are pairwise distinct, which implies $x\in E$. Finally, since 
digits from the seed set lie in $S_{\ast}\subset S$ and the inserted digits lie in $U\setminus S_{\ast}\subset S$, we have $x\in \{z\in E \colon A_n(z)\in S \text{ for all } n\ge 1\}.$
\end{proof}

Since \eqref{eq-sec4-3} and \eqref{eq-sec4-4} imply the
hypotheses \eqref{condi-Lip-1}, \eqref{condi-Lip-2} of Proposition~\ref{pro-almLip}, the
corresponding elimination map
\[f:E_{S,U}\to R_t(S_{\ast})\]
is almost Lipschitz. Hence, Lemma~\ref{Holder} and Proposition~\ref{seed-Prop} give
\[\dim_{\rm H} E_{S,U}\ge \dim_{\rm H} R_t(S_{\ast})=\frac{1}{2\alpha}.\]
Since $\deg A_n(x)\to \infty$ as $n\to \infty$ for any $x\in E$, by \eqref{hougan} and
Proposition~\ref{dim_Hirst},
\[\dim_{\rm H} E_{S,U}\le \dim_{\rm H}\{x\in E \colon A_n(x)\in S \text{ for all } n\ge 1\}\le\frac{1}{2\alpha}.\]
Hence,
\begin{equation}\label{eq-sec4-5}
\dim_{\rm H} E_{S,U}=\dim_{\rm H}\{x\in E\colon \{A_n(x)\colon n\in\mathbb N\}\subset S\}=\frac{1}{2\alpha}.
\end{equation}
We now verify the density statement. Put
\[D_{S,U}:=
\bigcup_{n\in\mathbb N}\ \bigcap_{x\in E_{S,U}}\{A_n(x)\}.\]
By construction, every element of each $W_k$ is inserted, in the same order and at the same prescribed position, into the digit sequence of every point of $E_{S,U}$. Hence, for each $k\ge 1,$
\[\bigcup_{j=1}^k W_j \subset D_{S,U}\cap U.\]
Using \eqref{eq-sec4-2}, for each $k\ge 1$ 
\[Q_{N_k}(U\setminus S_{\ast})=\bigcup_{j=1}^k W_j\subset D_{S,U}\cap U\]
and hence
\[\#Q_{N_k}(U\setminus S_{\ast})\le \#Q_{N_k}(D_{S,U}\cap U).\]
Dividing by $\#Q_{N_k}(S)$ and using \eqref{eq-sec4-1}, we obtain
\[\overline{d}_q(D_{S,U}\cap U | S)
\ge\limsup_{k\to\infty}\frac{\#Q_{N_k}(D_{S,U}\cap U)}{\#Q_{N_k}(S)}\ge
\lim_{k\to\infty}\frac{\#Q_{N_k}(U\setminus S_{\ast})}{\#Q_{N_k}(S)}=
\overline{d}_q(U| S).\]
Since $D_{S,U}\cap U\subset U$, the reverse inequality
\[\overline{d}_q(D_{S,U}\cap U | S)\le \overline{d}_q(U| S)\]
is obvious. Hence,
\begin{equation}\label{eq-sec4-6}
\overline{d}_q\left(
\bigcup_{n\in\mathbb N} \bigcap_{x\in E_{S,U}}\{A_n(x)\}\cap U \middle| S
\right)=\overline{d}_q(U| S).
\end{equation}
By construction of $U_{\deg}$ and $S_{\ast}\subset \tilde{S}=S\setminus U_{\deg}$, for any $d\in \deg(U)$ there exists $A\in U\setminus S_{\ast}$ such that $d=\deg(A),$ which implies
\begin{equation*}
\deg(U)\subset \deg(U\setminus S_{\ast})\subset \deg(D_{S, U}\cap U)\subset \deg(U).
\end{equation*}
This yields 
\begin{equation}
\label{denistydeg}
\overline{d}\left(
\bigcup_{n\in\mathbb N} \bigcap_{x\in E_{S,U}}\{\deg A_n(x)\}\cap \deg(U) \middle| \deg(S)
\right)=\overline{d}(\deg(U)| \deg(S)).
\end{equation}
Combining \eqref{eq-sec4-5},
\eqref{eq-sec4-6} and \eqref{denistydeg}, we complete the proof of Main Theorem.

%\subsection{Proof of Corollary~\ref{cor-FS}}

%Apply Main Theorem with
%$S=\mathbb F_q^1[X].$
%Since
%$\#Q_N(F_q^1[X]) = q^{N+1}-q,$
%the set $F_q^1[X]$ has growth density with exponent $\alpha=1$. Hence, for
%every subset $S\subset \mathbb F_q^1[X]$ with $\overline{d}_q(S)>0$, Main Theorem yields there exists a subset
%$E_S\subset E$ such that
%\[\dim_{\rm H} E_S=\dim_{\rm H} E=\frac{1}{2},\]
%\[\overline{d}_q\left(
%\bigcup_{n\in\mathbb N}\ \bigcap_{x\in E_S}\{A_n(x)\}\cap S
%\right)=\overline{d}_q(S),\]and 
%\begin{equation*}\overline{d}\left(\bigcup_{n\in\mathbb N}\bigcap_{x\in E_{S}}\{\deg A_n(x)\}\cap \deg(S)\right)=\overline{d}(\deg(S)),\end{equation*}
%which completes the proof of Corollary~\ref{cor-FS}.

\subsection*{Acknowledgments}  The author thanks Hiroki Takahasi for valuable comments. The author thanks  Kaoru Sano for a question that led to this study. The author thanks Reimi Irokawa for kind guidance on non-Archimedean dynamics. This research was supported by the JSPS KAKENHI 25K17282, Grant-in-Aid for Early-Career Scientists.


\begin{thebibliography}{99}

\bibitem{Ar}
E. Artin,
Quadratische K\"orper im Gebiete der H\"oheren Kongruenzen I--II,
{\it Math. Z.} {\bf 19} (1924), 153--246.
\bibitem{BLM}V. Bergelson, A. Leibman, and R. McCutcheon, Polynomial Szemerédi theorems for countable modules over integral domains and finite fields, {\it J. Analyse Math.} {\bf 95} (2005), no. 1, 243--296.
\bibitem{BN}
V. Berth\'e and H. Nakada,
On continued fraction expansions in positive characteristic: equivalence relations and some metric properties,
{\it Expo. Math.} {\bf 18} (2000), no. 4, 257--284.

\bibitem{Fal97}
K. Falconer,
{\it Techniques in Fractal Geometry},
John Wiley \& Sons, Ltd., Chichester, 1997.

\bibitem{Fal14}
K. Falconer,
{\it Fractal Geometry: Mathematical Foundations and Applications},
3rd ed., John Wiley \& Sons, Ltd., Chichester, 2014.

\bibitem{FSZ24}
Y. Feng, S. Shi, and Y. Zhang, Metrical properties for the weighted sums of degrees of
multiple partial quotients in continued fractions of Laurent series, {\it Finite Fields Appl.}
93 (2024), 102317.

\bibitem{Fur77}
H. Furstenberg,
Ergodic behavior of diagonal measures and a theorem of Szemer\'edi on arithmetic progressions,
{\it J. Analyse Math.} {\bf 31} (1977), 204--256.

\bibitem{Fur81}
H. Furstenberg,
{\it Recurrence in Ergodic Theory and Combinatorial Number Theory},
M. B. Porter Lectures, Princeton University Press, Princeton, NJ, 1981.

%\bibitem{FK78}
%H. Furstenberg and Y. Katznelson,
%An ergodic Szemer\'edi theorem for commuting transformations,
%{\it J. Analyse Math.} {\bf 34} (1978), 275--291.

\bibitem{GOQZ}
Z. Guo, K. Ouyang, J. Qiu, and S. Zhang,
Fractal transference principles for subsets of $\mathbb{N}^d$ of positive density,
{\it arXiv}:2601.14418.

\bibitem{G}
I. J. Good,
The fractional dimensional theory of continued fractions,
{\it Proc. Cambridge Philos. Soc.} {\bf 37} (1941), 199--228.

\bibitem{GT08}
B. Green and T. Tao,
The primes contain arbitrarily long arithmetic progressions,
{\it Ann. of Math.} {\bf 167} (2008), no. 2, 481--547.

\bibitem{HH}
D. Hu and X. Hu,
Arbitrarily long arithmetic progressions for continued fractions of Laurent series,
{\it Acta Math. Sci. Ser. B (Engl. Ed.)} {\bf 33} (2013), no. 4, 943--949.
\bibitem{HHY}
H. Hu, M. Hussain, and Y. Yu, Metrical properties for continued fractions of formal Laurent series, {\it Finite Fields Appl.} 73 (2021), 101850.
\bibitem{HS}
M. Hussain and N. Shulga, Hausdorff dimension for sets of continued fractions of formal Laurent series, {\it Finite Fields Appl.} 95 (2024), 102377.
\bibitem{HW}
X. Hu and J. Wu,
Continued fractions with sequences of partial quotients over the field of Laurent series,
{\it Acta Arith.} {\bf 136} (2009), no. 3, 201--211.

\bibitem{HWWY}
X. Hu, B. Wang, J. Wu, and Y. Yu,
Cantor sets determined by partial quotients of continued fractions of Laurent series,
{\it Finite Fields Appl.} {\bf 14} (2008), 417--437.

%\bibitem{Kr}
%S. Kristensen,
%On well-approximable matrices over a field of formal series,
%{\it Math. Proc. Cambridge Philos. Soc.} {\bf 135} (2003), 255--268.
\bibitem{Le}T. L\^e, 
Green--Tao theorem in function fields, {\it Acta Arith.} {\bf 147} (2011), no. 2, 129--152.
\bibitem{LN}R. Lidl and H. Niederreiter, Finite Fields, 2nd ed., Encyclopedia of
Mathematics and its Applications, Vol. 20, Cambridge University Press,
Cambridge, 1997.


\bibitem{NT}
Y. Nakajima and H. Takahasi,
Density combinatorics theorems in fractal dimension theory of continued fractions,
{\it Adv. Math.} {\bf 482} (2025), 110635.

\bibitem{NT2}
Y. Nakajima and H. Takahasi,
Multidimensional fractal transference principle for conformal iterated function systems,
in preparation.

\bibitem{NTW}
Y. Nakajima, H. Takahasi, and B. Wang,
Hausdorff dimension of sets of numbers whose continued fractions contain arbitrarily long arithmetic progressions,
{\it arXiv}:2601.12737.

\bibitem{Ni}
H. Niederreiter,
The probabilistic theory of linear complexity,
in {\it Advances in Cryptology---EUROCRYPT'88},
Lecture Notes in Comput. Sci., Vol. 330, Springer, New York, 1988, pp. 191--209.

%\bibitem{NV}
%H. Niederreiter and M. Vielhaber,
%Linear complexity profiles: Hausdorff dimensions for almost perfect profiles and measures for general profiles,
%{\it J. Complexity} {\bf 13} (1997), 353--383.

\bibitem{R85}
G. Ramharter,
Eine Bemerkung \"uber gewisse Nullmengen von Kettenbr\"uchen,
{\it Ann. Univ. Sci. Budapest. E\"otv\"os Sect. Math.} {\bf 28} (1985), 11--15.

\bibitem{S75}
E. Szemer\'edi,
On sets of integers containing no $k$ elements in arithmetic progression,
{\it Acta Arith.} {\bf 27} (1975), 199--245.

\bibitem{TW}
X. Tong and B. Wang,
How many points contain arithmetic progressions in their continued fraction expansion?
{\it Acta Arith.} {\bf 139} (2009), no. 4, 369--376.

\bibitem{We}A. Weingartner, On the degrees of polynomial divisors over finite fields,
{\it Math. Proc. Cambridge Philos. Soc.} {\bf 161} (2016), no. 3, 469–487.

\bibitem{Wae27}
B. L. van der Waerden,
Beweis einer Baudetschen Vermutung,
{\it Nieuw Arch. Wiskd.} {\bf 15} (1927), 212--216.

\bibitem{Wu}
J. Wu,
On the sum of degrees of digits occurring in continued fraction expansions of Laurent series,
{\it Math. Proc. Cambridge Philos. Soc.} {\bf 138} (2005), 9--20.

\bibitem{Wu2}
J. Wu,
Hausdorff dimensions of bounded type continued fraction sets of Laurent series,
{\it Finite Fields Appl.} {\bf 13} (2007), 20--30.

\bibitem{ZC2}
Z. Zhang and C. Cao,
On points with positive density of the digit sequence in infinite iterated function systems,
{\it J. Aust. Math. Soc.} {\bf 102} (2017), no. 3, 435--443.
 
 
 
 








\end{thebibliography}
\end{document}